\documentclass[9pt,a4paper, leqno]{amsart}

\usepackage{amssymb,amsmath,amsfonts,amsthm,a4}
\usepackage{graphicx, amsbsy, mathrsfs, esint, bbm, dsfont, bm}
\usepackage{color}
\usepackage{mathtools}
\usepackage{eucal}
\usepackage{slashed}

\usepackage{cite}
\usepackage{enumerate}

\usepackage[titletoc,title]{appendix}

\usepackage{mathrsfs} 

\usepackage{graphicx}   
\usepackage{subfigure}  
\usepackage{multirow}
\usepackage{slashed} 

\usepackage{xcolor}
\usepackage{hyperref}
\hypersetup{	
colorlinks=false, 
linkbordercolor={1 0 0},
citebordercolor={0 1 0},
}
\newtheorem{theorem}{Theorem}

\newtheorem{corollary}{Corollary}

\newtheorem*{assumption*}{Assumption}
\theoremstyle{remark}

\newtheorem*{remarks*}{Remarks}
\newtheorem*{remark*}{Remark}

\newcommand{\R}{\mathbb{R}} 
\newcommand{\C}{\mathbb{C}} 
\newcommand{\N}{\mathbb{N}} 
\newcommand{\T}{\mathbb{T}}
\newcommand{\HH}{\mathbb{H}}
\newcommand{\eps}{\varepsilon} 

\newcommand{\comment}[1]{}
\numberwithin{equation}{section}

\newcommand{\be}{\begin{equation}}
\newcommand{\ee}{\end{equation}}
\newcommand{\bes}{\begin{equation*}}
\newcommand{\ees}{\end{equation*}}

\newcommand{\Svec}{\vec{S}}
\newcommand{\Ss}{\mathbb{S}}
\newcommand{\Sb}{{\bm{S}}}
\newcommand{\pt}{\partial}

\newcommand{\ii}{\mathrm{i}}

\newcommand{\Tr}{\mathrm{Tr}}

\catcode`@=11
\def\section{\@startsection{section}{1}%
  \z@{1.5\linespacing\@plus\linespacing}{.5\linespacing}%
  {\normalfont\bfseries\large\centering}}
\catcode`@=12

\begin{document}

\title[Lax Pair for the Half-Wave Maps Equation]{A Lax Pair Structure \\ for the Half-Wave Maps Equation}

\author[P. G\'erard]{Patrick G\'erard}
\address{Laboratoire de Math\'ematiques d'Orsay, Univ. Paris-Sud, CNRS, Universit\'e Paris-Saclay, F-91405 Orsay, France.}
\email{patrick.gerard@math.u-psud.fr}

\author[E. Lenzmann]{Enno Lenzmann}
\address{University of Basel, Department of Mathematics and Computer Science, Spiegelgasse 1, CH-4051 Basel, Switzerland.}
\email{enno.lenzmann@unibas.ch}

\maketitle

\begin{abstract}
We consider the half-wave maps equation
$$
\pt_t \vec{S} = \vec{S} \wedge |\nabla| \vec{S},
$$
where $\vec{S}= \vec{S}(t,x)$ takes values on the two-dimensional unit sphere $\Ss^2$ and $x \in \R$ (real line case) or $x \in \T$ (periodic case). This an energy-critical Hamiltonian evolution equation recently introduced in \cite{LS,Zh}, which formally arises as an effective evolution equation in the classical and continuum limit of Haldane--Shastry quantum spin chains. We prove that the half-wave maps equation admits a Lax pair and we discuss some analytic consequences of this finding.  As a variant of our arguments, we also obtain a Lax pair for the half-wave maps equation with target $\HH^2$ (hyperbolic plane).
\end{abstract}


\section{Introduction and Main Results}

Spin chains -- both in quantum and classical versions -- arise as fundamental models in the study of exactly solvable and completely integrable systems. For instance, the classical Heisenberg model (HM) for ferromagnets in one space dimension provides a prototype of a completely integrable classical spin system; see \cite{La, Ta, FT}.

In this note, we are concerned with a new evolution equation for classical spins recently introduced in \cite{LS, Zh}, which we will refer to as the {\bf half-wave maps equation} following \cite{LS}. This equation has some similarities to (HM) and yet it shows a completely different mathematical features in many aspects (e.\,g., traveling solitary waves given by rational functions and energy-criticality of the evolution problem). In fact, the half-wave maps equation can be -- formally, at least -- obtained from taking a combined classical and continuum limit from a quantum spin chain of Haldane--Shastry (HS) type introduced in \cite{Ha,Sh}. Our main result shown below will yield a Lax pair for the half-wave maps equation, which will involve certain suitable {\em nonlocal} operators.

Let us now introduce the mathematical framework to formulate the problem at hand. We consider a time-dependent field of classical spins $\vec{S}= \vec{S}(t,x) \in \R^3$ defined for either $x \in \R$ (real line case) or $x \in \T=\R/2 \pi \mathbb{Z}$ (periodic case). Without loss of generality, we assume that the classical spins are normalized to unit length, i.\,e., we have that $\vec{S}(t,x) \in \Ss^2$ holds. For the spin field $\vec{S}=\vec{S}(t,x)$, we consider the half-wave maps equation given by
\be \tag{HWM} \label{eq:hwm}
\pt_t \Svec = \Svec \wedge |\nabla| \Svec  .
\ee  
Here $\wedge$ denotes the usual vector product in $\R^3$. The (pseudo-differential) operator $|\nabla|$ is defined via its corresponding multiplication symbol in Fourier space, i.\,e.,
$$
\widehat{(|\nabla| f)} = \begin{dcases*} |\xi| \widehat{f}(\xi) & ($x \in \R$, real line case) \\ |n| \widehat{f}_n & ($x \in \T$, periodic case). \end{dcases*}
$$
Here $\widehat{\cdot}$ denotes the Fourier transform for functions either defined on $\R$ or $\T$, respectively. The formal derivation of \eqref{eq:hwm} from a quantum spin chain model of Haldane--Shastry type will be sketched in Section \ref{sec:hwm-summary} below, followed by a brief summary about the traveling solitary waves for \eqref{eq:hwm} recently studied in \cite{LS}.

In analogy to the classical Heisenberg model, the half-wave maps equation comes with a Hamiltonian structure where energy functional in our case reads
\be
E[\vec{S}] =  \frac 1 2 \int \vec{S} \cdot |\nabla| \vec{S} \, dx  .
\ee
The corresponding Poisson bracket for the $\Ss^2$-valued function $\vec{S}=(S_1, S_2, S_3)$ is given by 
\be
\left \{ S_i(x), S_j(y) \right \} = \eps_{ijk} S_k(x) \delta(x-y),
\ee
where $\eps_{ijk}$ denotes the standard anti-symmetric Levi--Civit\`a symbol. As a consequence, the \eqref{eq:hwm} can be (formally) written as $\pt_t \Sb = \{ \Sb, E \}$ in analogy to the Heisenberg model. Furthermore, it is straightforward to check that \eqref{eq:hwm} exhibits formal conservation of total spin and linear momentum, due to the rotational invariance (on the target $\Ss^2$) and translational invariance on the domain; see \cite{LS} for details.

From the point of view of PDE analysis, the evolution problem \eqref{eq:hwm} is {\em energy-critical} since the scaling transform 
$$
\vec{S}(t,x) \mapsto \vec{S}_\lambda(t,x) = \vec{S}(\lambda t, \lambda x)
$$ 
maps solutions into solutions, whereas the energy $E[\vec{S}_\lambda] = E[\vec{S}]$ stays invariant under this transformation. Such a critical scaling behavior is in striking contrast to the one-dimensional Heisenberg model (HM), which is energy-subcritical. As a consequence, the existence of unique global-in-time solutions for the half-wave maps equation is much more delicate. In particular, a possible singularity formation (blowup) of smooth solutions cannot be simply ruled out by using energy conservation.

Let us now turn to the issue of complete integrability for the half-wave maps equation, where we will show below that \eqref{eq:hwm} admits a Lax pair. As in the study of the classical Heisenberg model, it turns out to be expedient to first formulate the problem by using the standard Pauli matrices $\sigma_1, \sigma_2, \sigma_3 \in \mathfrak{su}(2)$. For a given vector $\vec{X} \in \R^3$, we define the $2\times2$--matrix given by
$$
\bm{X} = \vec{X} \cdot \bm{\sigma} = \sum_{j=1}^3 X_j \sigma_j = \left ( \begin{array}{cc} X_3 & X_1 - \ii X_2 \\ X_1 + \ii X_2 & -X_3 \end{array} \right ).
$$ 
with the usual notation $\bm{\sigma}=(\sigma_1, \sigma_2, \sigma_3)$. Clearly, we have that $\bm{X}^* = \bm{X}$ is a Hermitian $2 \times 2$--matrix with vanishing trace $\Tr \, \bm{X}=0$. For the reader's convenience, we recall some standard facts when dealing with Pauli matrices. From the fundamental relation $\sigma_{j} \sigma_k = \delta_{jk} \mathds{1} + \ii \eps_{jkl}\sigma_l$ (where $\mathds{1}$ denotes $2\times2$--unit matrix) we obtain
\be \label{eq:pauli_product}
( \vec{X} \cdot \bm{\sigma}) (\vec{Y} \cdot \bm{\sigma}) = (\vec{X} \cdot \vec{Y}) \mathds{1} + \ii (\vec{X} \wedge \vec{Y}) \cdot \bm{\sigma} .
\ee
for arbitrary vectors $\vec{X}, \vec{Y} \in \R^3$. As a direct consequence of this, we readily deduce that
\be \label{eq:S_unitlength}
\bm{S}^2 = \mathds{1} \quad \mbox{if and only if} \quad |\vec{S}|^2=1.
\ee
Furthermore, we see that \eqref{eq:hwm} can be equivalently written as 
\be \label{eq:hwm_S}
\pt_t \Sb = -\frac{\ii}{2} [\Sb, |\nabla| \Sb],
\ee
where $[A,B] = AB-BA$ denotes the commutator of $A$ and $B$. Clearly, this matrix-valued formulation as a commutator equation bears a strong resemblance to the equation $\pt_t \Sb= - \frac{\ii}{2} [\Sb, \frac{d^2}{dx^2} \Sb]$ used to prove the complete integrability of the one-dimensional Heisenberg model (HM) by Takthajan in \cite{Ta}. However, the presence of the nonlocal pseudo-differential operator $|\nabla|$ in equation \eqref{eq:hwm_S} makes the search for a Lax pair a rather different task.

To formulate the operators for the Lax pair of the half-wave maps equation, we introduce the following notation in order to avoid any potential ambiguities in the expressions below. Suppose that $\vec{A}=\vec{A}(t,x) \in \R^3$ is a given function and let $\bm{A} = \vec{A} \cdot \bm{\sigma}$ denote the corresponding function with values in the Lie algebra $\mathfrak{su}(2)$. We use $\mu_{\bm{A}}$ to denote the multiplication operator for $\bm{A}$ acting on functions $\varphi = \varphi(x)$ with values in $\C^2$, i.\,e., we set
$$
(\mu_{\bm{A}} \varphi)(t,x) = \bm{A}(t,x) \varphi(x).
$$
Given a spin field $\vec{S}=\vec{S}(t,x)$, we are now ready to define the following pair of operators $L_{\Sb}$ and $B_{\Sb}$ as follows
\be \label{def:Lax}
 L_\Sb = [H, \mu_{\Sb} ]  \quad \mbox{and} \quad B_\Sb = - \frac{\ii}{2} \left ( \mu_{\Sb} |\nabla| + |\nabla| \mu_{\Sb} \right ) + \frac{\ii}{2} \mu_{|\nabla| \Sb} .
\ee
The operators $L_\Sb$ and $B_\Sb$ are formally defined on the complex Hilbert space $L^2(X; \C^2)$, where either $X=\R$ (real line case) or $X = \T$ (periodic case). Furthermore, the operator $H$ denotes the {\em Hilbert transform} defined through the principal value expression given by 
$$
(Hf)(x) = \begin{dcases*} \mathrm{p.v.} \frac{1}{\pi} \int_\R \frac{f(y)}{x-y} dy & ($x \in \R$, real line case) \\
\mathrm{p.v.} \int_{\T} f(y) \cot \left ( \frac{x-y}{2} \right ) \, dy & ($x \in \T$, periodic case).
 \end{dcases*}
$$ 
In \eqref{def:Lax} above and throughout the following, the Hilbert transform $H$ has to be understood to act as $H \varphi = (H \varphi_1, H \varphi_2)$ on each of the components of $\varphi \in L^2(X; \C^2)$. Since $H^* = -H$ is skew-symmetric and $(\mu_{\bm{A}})^*= \mu_{\bm{A}}$ is symmetric, we readily deduce the (formal) properties  
$$
(L_\Sb)^* = L_{\Sb} \quad \mbox{and} \quad (B_\Sb)^* = -B_{\Sb}.
$$ 
In fact, we will see below that $L_\Sb$ is a Hilbert--Schmidt operator if and only if $\vec{S}$ has finite energy. On the other hand, the operator $B_\Sb$ is clearly an unbounded operator to be defined on some suitable dense subset of $L^2(X; \C^2)$.

The main result of this paper shows that $L_\Sb$ and $B_\Sb$ provide  a Lax pair for the half-wave maps equation.

\begin{theorem}[Lax Pair of the Half-Wave Maps Equation] \label{thm:lax}
Let  $\vec{S} = \vec{S}(t,x)$ be a sufficiently regular solution of \eqref{eq:hwm} with either $x \in \R$ (real line case) or $x\in \T$ (periodic case). Then the following Lax equation holds true:
$$
\frac{d}{dt} {L}_\Sb = [B_\Sb, L_{\Sb}] ,
$$
where the operators $L_\Sb$ and $B_\Sb$ are defined in \eqref{def:Lax} above.
\end{theorem}

\begin{proof}
We first recall that $d\mu_{\Sb}/dt = -\frac{\ii}{2} [\mu_{\Sb}, \mu_{|\nabla| \Sb}]$ from \eqref{eq:hwm_S}. By using Jacobi's identity $[A,[B,C]]+[C,[A,B]]+[B,[C,A]] \equiv 0$, we get
\be
\frac{d}{dt}{L}_\Sb = -\frac{\ii}{2} [H, [\mu_{\Sb}, \mu_{|\nabla| \Sb}]] = \frac{\ii}{2} \left ( [\mu_\Sb,[\mu_{|\nabla| \Sb}, H]] + [\mu_{|\nabla| \Sb},[H, \mu_{\Sb}]] \right ) .
\ee
Next, from Cotlar's product identity $H(fg) = Hf g + f Hg + H(Hf Hg)$ for the Hilbert transform we deduce
\be
[\mu_{|\nabla| \Sb}, H] = - \left ( \mu_{H|\nabla| \Sb} + H \mu_{H |\nabla| \Sb} H \right ) = \mu_{\pt_x \Sb} + H \mu_{\pt_x \Sb} H,
\ee
where we also used that $H |\nabla| = -\pt_x$ holds. Hence can write the time derivative of $L_{\Sb}$ as
\be \label{eq:magic0}
\frac{d}{dt} {L}_\Sb= \frac{\ii}{2} [\mu_{\Sb}, \mu_{\pt_x \Sb} + H \mu_{\pt_x \Sb} H] + \frac{\ii}{2} [\mu_{|\nabla| \Sb}, L_{\Sb}].
\ee
It remains to show that the first commutator on right-hand side can be written as a commutator with $L_{\Sb}$. This can be seen as follows. Using $|\nabla|=H\pt_x$ again, we get
\be \label{eq:magic1}
[\mu_{\Sb} |\nabla| , \mu_{\Sb} H] = [ \mu_{\Sb} H \pt_x, \mu_{\Sb}  H] = \mu_{\Sb} H [\pt_x, \mu_{\Sb}] H = \mu_{\Sb} H \mu_{\pt_x \Sb} H,
\ee  
where in the last step we used $[\pt_x, \mu_{\Sb}] = \mu_{\pt_x \Sb}$ by Leibniz' rule. On the other hand, in view of $\mu_{\Sb} \mu_{\Sb} = \mathds{1}$ by \eqref{eq:pauli_product} and \eqref{eq:S_unitlength} together with $H \pt_x H = H^2 \pt_x = -\pt_x$, we see that
\begin{align*}
[\mu_{\Sb} |\nabla|, H \mu_{\Sb}] &  = [\mu_{\Sb} H \pt_x, H \mu_{\Sb}]  = \mu_{\Sb} H \pt_x H \mu_{\Sb} - H \mu_{\Sb} \mu_{\Sb} H \pt_x \\
&  = -\mu_{\Sb} \pt_x \mu_{\Sb} + \pt_x =  -\mu_{\Sb} [ \pt_x, \mu_{\Sb}] = -\mu_{\Sb} \mu_{\pt_x \Sb}.
\end{align*}
Recall now the definition $L_{\Sb} = H \mu_{\Sb} - \mu_{\Sb} H$.  Thus if we now combine \eqref{eq:magic1} with the identity found above, we deduce the identity
\be \label{eq:magic2}
[ \mu_{\Sb} |\nabla|, L_{\Sb}] = - \mu_{\Sb} \mu_{\pt_x \Sb} - \mu_{\Sb} H \mu_{\pt_x \Sb} H 
\ee
Since $(L_{\Sb})^*= L_{\Sb}$, we can take adjoints to calculate $[ |\nabla| \mu_{\Sb}, L_{\Sb}] $. However, it is also interesting to make the computation directly, namely
$$[|\nabla|\mu_{\Sb}  , H\mu_{\Sb}] = [ \pt_x H \mu_{\Sb} , H\mu_{\Sb}  ] = [\pt_x, H\mu_{\Sb}] H  \mu_{\Sb} = H \mu_{\pt_x \Sb} H \mu_{\Sb}\ ,$$
and
\begin{align*}
[ |\nabla|\mu_{\Sb}, \mu_{\Sb}H] &  = [H \pt_x\mu_{\Sb} , \mu_{\Sb}H]  =  H \pt_x\mu_{\Sb} \mu_{\Sb} H-\mu_{\Sb} H \pt_x H \mu_{\Sb} \\
&  = -\pt_x  +\mu_{\Sb} \pt_x \mu_{\Sb}=  \mu_{\Sb} [ \pt_x, \mu_{\Sb}] =  \mu_{\Sb}\mu_{\pt_x \Sb}=-\mu_{\pt_x \Sb}\mu_{\Sb},
\end{align*}
where last identity follows from differentiating $\mu_{\Sb} \mu_{\Sb}  = \mathds{1}$. Thus we conclude that
\be \label{eq:magic3}
[ |\nabla| \mu_{\Sb}, L_{\Sb}] = \mu_{\pt_x \Sb} \mu_{\Sb} + H \mu_{\pt_x \Sb} H \mu_{\Sb}. 
\ee
By adding \eqref{eq:magic2} and \eqref{eq:magic3}, we arrive at
$$
[ \mu_{\Sb} |\nabla| + |\nabla| \mu_{\Sb}, L_{\Sb}] = - [ \mu_{\Sb}, \mu_{\pt_x \Sb} + H \mu_{\pt_x \Sb} H ]. 
$$
If we now recall the definition of the operator 
$$
B_{\Sb} = -\frac{\ii}{2} ( \mu_{\Sb} |\nabla| + |\nabla| \mu_{\Sb}) + \frac{\ii}{2} \mu_{|\nabla| \Sb},
$$
 we deduce from \eqref{eq:magic0} that $dL_{\Sb}/dt = [B_{\Sb}, L_{\Sb}]$ holds, as desired.

The proof of Theorem \ref{thm:lax} is now complete. 
\end{proof}

We now discuss some direct consequences of Theorem \ref{thm:lax}.

\begin{corollary}
We have the following (formal) conservation laws:
$$
\Tr ( |L_\Sb|^p) = \mathrm{const.}
$$	
for any $1 \leq p < \infty$.
\end{corollary}

\begin{remarks*}
1) Consider the real line case when $x \in \R$. Then the operator $L_\Sb$ has the symmetric kernel
$$
K_\Sb(x,y) = \frac{1}{\pi} \frac{\Sb(x)-\Sb(y)}{x-y} \in \C^{2 \times 2} .
$$
Using that $\Tr_{\C^2} ( \bm{A} \bm{A}^*) = 2 |\vec{A}|^2$ thanks to \eqref{eq:pauli_product}, we find that the squared Hilbert--Schmidt norm of the Lax operator $L_\Sb$ is given by
\be \label{eq:TrL2_En}
\Tr ( |L_\Sb|^2) = \frac{2}{\pi^2} \int \! \! \int_{\R \times \R} \frac{|\vec{S}(x)- \vec{S}(y)|^2}{|x-y|^2} \,dx \,dy = \frac{8}{\pi} E[\vec{S}],
\ee
where we used the well-known identity 
$$
\int_{\R} f |\nabla| f \, dx = \frac{1}{2 \pi} \int \! \! \int_{\R \times \R} \frac{|f(x)-f(y)|^2}{|x-y|^2} dx \, dy.
$$
Hence $\Tr(|L_{\Sb}|^2)$ is equal to the conserved energy $E[\vec{S}]$ for \eqref{eq:hwm} (up to a multiplicative constant). Moreover, we deduce that $L_{\Sb}$ is Hilbert--Schmidt if and only if $E[\vec{S}] < +\infty$. 

2) In the periodic case when $x \in \T$, an analogous calculation yields that
\begin{align*}
\Tr (|L_{\Sb}|^2) & = 2 \int \! \! \int_{\T \times \T} |\vec{S}(x)-\vec{S}(y)|^2 \cot^2 \left ( \frac{x-y}{2} \right ) dx \, dy \\
& = 2 \int \! \! \int_{\T \times \T} \frac{|\vec{S}(x)-\vec{S}(y)|^2}{\sin^2(\frac{1}{2} (x-y))} \, dx \, dy - 2 \int \! \! \int_{\T \times \T} |\vec{S}(x)-\vec{S}(y)|^2 \, dx \, dy \\
& = 8 \pi E[\vec{S}] + 4 \left | \int_{\T} \vec{S}(x) \, dx \right |^2 - 16 \pi^2,
\end{align*}
which amounts to a combination of the conserved energy and the square of the conserved total spin $\int_{\T} \vec{S}$, plus some numerical constant. Again, we see that $L_\Sb$ is Hilbert-Schmidt if and only if $\vec{S}$ has finite energy.

3) More generally, by decomposing $L^2(X,\C )=L^2_+(X,\C )\oplus L^2_-(X,\C)$ according to the sign of the Fourier spectrum, it is easy to make the link between operators $L_{\Sb}$ and Hankel operators with matrix symbols given by symbols $\Sb$; see Peller \cite{Pe2}. As a consequence from Peller's theorem \cite{Pe1}, we conclude (both in real line and the periodic case) that the following norm equivalence to homogeneous Besov norms of $\vec{S}$ holds:
\be 
\Tr(|L_{\Sb}|^p) \sim \| \vec{S} \|_{\dot{B}^{1/p}_{p,p}}^p .
\ee
In particular, applying this result for $p=1$, if the initial value of $\vec{S}$ is smooth enough, then $|\nabla |\vec{S}, \pt_x \vec{S}$ and $\pt_t\vec{S}$ are uniformly bounded in $L^1(X,\R^3)$ for $t$ in the interval of existence of the solution.

4) 
As an instructive example, let us explicitly compute the Lax operator for the profile $\vec{Q}_v : \R \to \Ss^2$ of a traveling solitary wave of degree $m=1$; see Section \ref{sec:hwm-summary} below for more details on solitary waves for the half-wave maps equation. As a traveling solitary wave profiles of degree $m=1$ with velocity $v \in \R$ and $|v| < 1$, we can take (without loss of generality) the following function
$$
\vec{Q}_v(x)= \left (\alpha_v f(x), \alpha_v g(x), v \right ) = \left ( \alpha_v \frac{x^2-1}{1+x^2}, \alpha_v \frac{-2x}{1+x^2}, v \right )  \quad \mbox{with $\alpha_v = \sqrt{1-v^2}$}.
$$
With the help of the singular integral expression for $H$, we verify the commutator formulas
$$
[H, f] u =  \langle \psi, u \rangle \varphi + \langle  \varphi, u \rangle \psi, \quad [H,g] u = - \langle \varphi, u \rangle \varphi + \langle \psi, u \rangle \psi, 
$$
with the functions $\varphi(x) = \sqrt{\frac{2}{\pi}} \frac{1}{1+x^2}$, $\psi(x) = \sqrt{\frac{1}{2 \pi}} \frac{2x}{1+x^2}$, and $\langle \cdot, \cdot \rangle$ being the scalar product on $L^2(\R;\C)$. Hence it follows that the corresponding Lax operator for $\vec{Q}_v$ is found to be
$$
L_{\bm{Q}_v}  = [H, \alpha_v f] \sigma_1 + [H, \alpha_v g] \sigma_2 + [H, v] \sigma_3 = \alpha_v \left [ \begin{array}{cc} 0 & [H,f] - \ii [H, g] \\  {[H,f] + \ii [H,g]} & 0 \end{array} \right ]
$$
acting on $L^2(\R; \C^2) = L^2(\R; \C) \oplus L^2(\R; \C)$. It is evident that the range of $L_{\bm{Q}_v}$ belongs to  the four-dimensional space spanned by orthonormal basis $\{  {\varphi \choose 0},  {\psi \choose 0},{0  \choose \varphi},  {0 \choose \psi}\}$. With respect to this basis, it is easy to check that $L_{\bm{Q}_v}$ has the corresponding matrix
$$
M = \alpha_v \left [ \begin{array}{rrrr} 0 & 0 &  +\ii & 1 \\ 0 & 0 & 1 & -\ii \\ -\ii & 1 & 0 & 0 \\ 1 & +\ii & 0 & 0  \end{array} \right ]  .
$$
Its eigenvalues (counting multiplicities)  are found to be $\lambda_1 = -2 \alpha_v, \lambda_2 = \lambda_3 = 0, \lambda_4 = +2 \alpha_v$. Thus the spectrum of the corresponding Lax operator of $\vec{Q}_v$ is found to be
$$
\mathrm{spec} \left ( L_{\bm{Q}_v}  \right ) = \{ 0, -2 \alpha_v, +2 \alpha_v \} \quad \mbox{with $\alpha_v = \sqrt{1-v^2}$}.
$$
Also note that $\mathrm{Tr} (|L_{\bm{Q}_v}|^2) = 8 \alpha_v^2 =(1-v^2) E[\vec{Q}_v]/\pi$, which is of course in accordance with relation \eqref{eq:TrL2_En} above and identity \eqref{eq:EQv} below.
\end{remarks*}

Using Kronecker's theorem characterizing finite rank Hankel operators \cite{Pe2}, another fundamental consequence of Theorem \ref{thm:lax} is as follows.

\begin{corollary}
Let $\vec{S}=\vec{S}(t,x)$ solve \eqref{eq:hwm} for every $t$ in an interval $I$ containing 0. 
\begin{itemize}
\item If $x\in \R$ and $\vec{S}(0,x)$ is a rational function of $x$, then so is $\vec{S}(t,x)$ for every $t\in I$.
\item If $x\in \T$ and $\vec{S}(0,x)$ is a rational function of ${\rm e}^{\ii x}$, then so is $\vec{S}(t,x)$ for every $t\in I$.
\end{itemize}
\end{corollary}

\begin{remark*}
In fact, the rank of $L_{\Sb}$ (which is constant in time by the Lax equation) can be used to bound the number of poles of the rational functions in the components of $\vec{S}(t,x)$. In particular, we expect multi-soliton solutions for the half-wave maps equation in the subset of rational solutions.
\end{remark*}

\subsection*{Acknowledgments} P.\,G. was supported by the French A.N.R. through grant ANAE 13-BS01-0010-03. E.\,L.~was supported by the Swiss National Science Foundation (SNF) through Grant No.~200021-169646. Furthermore, E.\,~L.~is grateful to E.~Langmann and M.~Stone for pointing out reference \cite{Zh}. The authors wish to thank for the kind hospitality of the Mathematisches Forschungsinstitut Oberwolfach (MFO) during their stay in June 2017. Finally, we thank the anonymous referee for useful comments on this paper.

\section{Link to Haldane--Shastry Models and Solitary Waves}

\label{sec:hwm-summary}

\subsection{Relation to Haldane--Shastry Models}

Following \cite{LS} and \cite{Zh}, we explain how the evolution equation \eqref{eq:hwm} can be formally obtained from a discrete quantum spin system of Haldane--Shastry (HS) type by taking the classical (large-spin) limit followed by taking the continuum limit. Let us also mention that HS-type model have a close connection to {\em Calogero--Moser--Sutherland models} which have been intensively studied in the past decades; see \cite{Su} for a review on these models.

To illustrate this procedure, we consider the periodic setting on $\T$. (The formal arguments here carry over to the real line case $\R$ with some modifications.) Let $N \geq 2$ be an integer and divide $\T$ by introducing equally spaced lattice points $x_j = 2\pi i k/N$ with $k=1, \ldots, N$ and $x_0 \equiv x_N \mod 2 \pi$. At each site $x_j \in \T$, we attach a quantum spin of size $s \in \frac{1}{2} \N$ (integer or half-integer) and we consider the quantum Hamiltonian $H_{HS}$ defined as 
\be
H_{HS} = \sum_{j < k}^N \frac{1 - \vec{\mathcal{S}} (x_j) \cdot \vec{\mathcal{S}}(x_k)}{\sin^2 [ \frac{1}{2} (x_j - x_k)]},
\ee
acting on the Hilbert space $\mathcal{H} = (\C^{2s +1})^{\otimes N}$ (the $N$-fold tensor product of $\C^{2s+1}$). Note that $\sin^{2} ( \frac{1}{2} (x_j-x_k)) = \frac{1}{2} |e^{\ii x_j}- e^{\ii x_k}|^2$ is proportional to the squared chord distance between the points $e^{\ii x_j}$ and $e^{\ii x_k}$ on the unit circle $\Ss^1$. Here $\vec{\mathcal{S}}(x_j) = (\mathcal{S}_1(x_j), \mathcal{S}_2(x_j), \mathcal{S}_3(x_j))$ is the quantum spin operator associated to site $x_j$, where its entries are given by the generators of the spin-$s$-representation of $\mathfrak{su}(2)$ acting on $\C^{2s+1}$, rescaled by $s^{-1}$ for later convenience. (For $s=1/2$, the operators $\mathcal{S}_i$ are given by the Pauli matrices $\sigma_i$ acting on $\C^2$.) We have the general commutation relations
\be
[\mathcal{S}_\alpha(x_j), \mathcal{S}_\beta(x_k)] =  \frac{\ii}{s} \eps_{\alpha \beta \gamma} \mathcal{S}_\gamma(x_j) \delta_{jk} .
\ee
In summary, the above Hamiltonian $H_{HS}$ defines a Haldane--Shastry type quantum spin chain with long-range $1/|x|^2$ interactions of ferromagnetic type (because aligning the spins in the same direction is energetically favorable).

Now  we study the {\em classical (large-spin) limit} by passing to $s \to +\infty$ (which can also viewed as a semi-classical limit with parameter $\hbar=s^{-1}\to 0$). In heuristic terms, this passage amounts to replacing the quantum spins by classical spin variables $\vec{S}(x_j) \in \Ss^2$, i.\,e., unit vectors in $\R^3$. We refer to e.\,g.~to \cite{FKL, Li} where semi-classical spin limits were rigorously studied in the context  spin dynamics (with smooth short-ranged interactions) and partition functions for spin systems, respectively. In summary, we (formally at least) obtain as a classical limit of $H_{HS}$ the following Hamiltonian
\be 
H_{HS}^{\mathrm{(classical)}} = \sum_{j < k}^N \frac{1 - \vec{S}(x_j) \cdot \vec{S}(x_k)}{\sin^2 [ \frac{1}{2} ( x_j - x_k) ]} ,
\ee
which is defined on the classical phase space $\Gamma = \prod_{x_j \in \T} \Ss^2$ (i.\,e.~ the $N$-fold cartesian product of $\Ss^2$ with itself). On the space $\Gamma$, we have the canonical Poisson bracket which reads
\be
\{ S_\alpha(x_j), S_\beta(x_k) \} = \eps_{\alpha \beta \gamma} S_\gamma(x_j) \delta_{jk}.
\ee
Using that $|\vec{S}(x_j)|^2=1$, we readily check that the equation of motions $\pt_t \vec{S}(t,x_k) = \{ \mathcal{S}(t,x_k), H_{HS}^{\mathrm{(classical)}} \}$ are found to be
\be \label{eq:HS_disc}
\pt_t \vec{S}(t, x_k) = \vec{S}(t,x_k) \wedge \left ( \sum_{j \neq k}^N \frac{\vec{S}(t,x_k) - \vec{S}(t,x_j)}{ \sin^2 [\frac{1}{2}(x_j - x_k)]} \right ) 
\ee
for every $k=1, \ldots, N$. Now if we pass to the continuum limit $N \to +\infty$ so that the lattices site $x_j \in \T$ range over all of $\Ss^2$, we formally arrive (after a suitable time rescaling $t \to \mbox{const} \, N^{-1} t$) at the half-wave maps equation
\be \label{eq:HS_cont}
\pt_t \vec{S} = \vec{S} \wedge |\nabla| \vec{S}
\ee
posed on $\T$, where we recall that the singular integral expression for $|\nabla|$ in the periodic setting is $(|\nabla| f)(x) = \frac{1}{4 \pi} \mathrm{p.v.} \int_{\T} \frac{f(x)-f(y)}{\sin^2((x-y)/2)} dy$ for $x \in \T$. A rigorous investigation of this continuum limit procedure passing from \eqref{eq:HS_disc} to \eqref{eq:HS_cont} will be addressed in \cite{BuLe}.

With regard to complete integrability, let us mention that the quantum Hamiltonian $H_{HS}$ is known to admit a (quantum) Lax pair; see \cite{Su} for a review on Haldane--Shastry models. For instance, if we take $s=1/2$, the (quantum) Lax operator $\widehat{L}$ has operator-valued entries that read
$$
\widehat{L}_{jk} = \ii (1-\delta_{jk}) \frac{\mathds{1} + \vec{\mathcal{S}}(x_j) \cdot \vec{\mathcal{S}}(x_k)}{x_j -x _l} \in \C^{2 \times 2} \quad \mbox{for $k,l=1, \ldots, N$}.
$$
However, it seems not to be a straightforward procedure (not even formally) to deduce the right expression for a Lax operator for the classical models given by $H_{HS}^{\mathrm{(classical)}}$ (discrete case) and the half-wave maps equation (continuum case). In fact, to the best of our knowledge, finding a Lax pair for the discrete classical spin chain given by $H^{\mathrm{(classical)}}_{HS}$ has not been achieved yet. A suitably discretized version of $L_{\Sb}$ seems to be a natural candidate for a Lax operator for $H_{HS}^{\mathrm{(classical)}}$.

\subsection{Traveling Solitary Waves} 
 
A remarkable fact recently found in \cite{LS,Zh} is that the half-wave maps equation admits non-trivial traveling solitary wave solutions 
\be
\vec{S}(t,x) = \vec{Q}_v(x- vt),
\ee 
where the parameter $v \in \R$ denotes the velocity.  It is easy to check that the profile $\vec{Q}_v=\vec{Q}_v(x)$ has to satisfy the nonlinear equation
\be \label{eq:Qv}
\vec{Q}_v \wedge |\nabla| \vec{Q}_v - v \pt_x \vec{Q}_v = 0 .
\ee
For the special case of vanishing velocity $v=0$, we obtain the so-called {\em half-harmonic maps equation} $\vec{Q} \wedge |\nabla| \vec{Q}=0$. In fact, this equation was recently introduced in \cite{LR}  by a completely different motivation coming from conformally invariant problems in the study of PDE.

From \cite{LS}  we recall the following explicit classification result for profiles $\vec{Q}_v : \R \to \Ss^2$ with finite energy.
\begin{itemize}
\item If $|v| < 1$, then any profile $\vec{Q}_v : \R \to \Ss^2$ with finite energy of the form
$$
\vec{Q}_v(x) = \left ( \sqrt{1-v^2} \, \mathrm{Re}  \, B(x), \sqrt{1-v^2} \, \mathrm{Im} \, B(x), vÊ\right )
$$
up to rotations on $\Ss^2$ and a complex conjugation symmetry. Here $B=B(x+\ii y)$ is a finite Blaschke product defined on the (closed) upper complex plane $\overline{\C}_+$, i.\,e., we have
$$
B(z) =  \prod_{k=1}^m \frac{ z - z_k}{z - \overline{z}_k},
$$ 
with some  $m \in \N$ and $z_1, \ldots, z_m \in \C_+$. Note that $m=0$ corresponds to the case of trivial constant profile $\vec{Q}_v$.

\item If $|v| \geq 1$, then any profile $\vec{Q}_v : \R \to \Ss^2$ with finite energy is trivial, i.\,e., 
$$
\vec{Q}_v(x) \equiv \vec{P}
$$
for some $\vec{P} \in \Ss^2$.
\end{itemize}
The arguments in \cite{LS} exploit a close connection to minimal surfaces inside the unit ball, where $\vec{Q}_v(x)$ arise as a boundary curve on $\Ss^2$. We remark that the dependence on $v$ in the expression for $\vec{Q}_v$ can be indeed be regarded as a Lorentz boost implemented by the conformal group (i.\,e.~the M\"obius group) acting on the target sphere $\Ss^2$; see \cite{LS}. Furthermore, the energy of $\vec{Q}_v$ is found to be quantized by multiples of $\pi$ such that
\be \label{eq:EQv}
E[\vec{Q}_v] = (1-v^2) \cdot \pi m .
\ee 
Thus there exist traveling solitary waves for \eqref{eq:hwm} with arbitrarily small energy. This is in stark contrast to other energy-critical dispersive geometric PDEs (e.\,g.~energy-critical Schr\"odinger maps, wave maps, and Yang--Mills equations) where a certain energy threshold exists, below which finite-energy solutions scatter to ``free'' solutions.

Finally, we refer to \cite{LS} for a complete spectral analysis of the linearized operator $\vec{Q}_v$ in the static case $v=0$.

\section{Extension of Results to Target $\HH^2$}

A variant of geometric interest of the half-wave maps equation occurs when the target two-sphere $\Ss^2$ is replaced by the hyperbolic plane $\HH^2$, which is a non-compact K\"ahler manifold. To formulate the corresponding evolution equation, we regard $\HH^2$ as embedded into Minkowsi three-space $\R^{1,2}$ as a unit pseudosphere with positive component $X_1  >0$, i.\,e., we set
\be
\HH^2 = \left \{ \vec{X} \in \R^{1,2} : -X_1^2+X_2^2+X_3^2 = -1, \; X_1 > 0  \right \} .
\ee 
Let $(\eta_{ij}) =\mathrm{diag}(-1,+1,+1)$ denote the Lorentzian metric on $\R^{1,2}$. We define the (non-definite) inner product
$$
\vec{X} \cdot_\eta \vec{Y} = (\eta \vec{X}) \cdot Y = -X_1 Y_1 + X_2 Y_2 + X_3 Y_3.
$$
Likewise, we introduce the cross-type product for vectors $\vec{X}, \vec{Y} \in \R^{1,2}$ by setting
\be
\vec{X} \wedge_{\eta} \vec{Y} = \eta ( \vec{X} \wedge \vec{Y} ) = (-(X_2 Y_3- X_3 Y_2), X_3 Y_1 - X_1 Y_3,  X_1 Y_2 - X_2 Y_1 ).
\ee
Now the half-wave maps equation on $\R$ or $\T$ with hyperbolic plane target $\HH^2$ is given by
\be \tag{HWM$_{\HH^2}$} \label{eq:hwmh}
\pt_t \vec{S} = \vec{S} \wedge_{\eta} |\nabla| \vec{S} .
\ee
This is a Hamiltonian equation with the corresponding conserved energy 
\be
E_\eta[\vec{S}] = \frac 1 2 \int \vec{S} \cdot_\eta |\nabla| \vec{S} = \frac{1}{2} \int \left ( -(S_1 |\nabla| S_1) + (S_2 |\nabla| S_2) + (S_3 |\nabla| S_3) \right ) .
\ee
Note that $E_\eta$ is not positive definite due to indefinite scalar product $\cdot_\eta$.

To show that this equation admits a Lax pair, we proceed as follows. Consider the matrices 
\be
\rho_1 = \left ( \begin{array}{cc} \ii & 0 \\ 0 & -\ii \end{array} \right ) , \quad \rho_2 = \left ( \begin{array}{cc} 0 & 1 \\ 1 & 0 \end{array} \right ), \quad \rho_3 = \left ( \begin{array}{cc} 0 & \ii \\ -\ii & 0 \end{array} \right ) ,
\ee
which span the Lie algebra $\mathfrak{su}(1,1)$. Now given a vector $\vec{X} \in \R^{1,2}$, we define the complex $2 \times 2$-matrix given by
\be
\bm{\tilde{X}} = \vec{X} \cdot \bm{\rho} = \sum_{j=1}^3 X_j \rho_j = \left ( \begin{array}{cc} \ii X_1 & X_2 + \ii X_3 \\ X_2 - \ii X_3 & - \ii X_1 \end{array} \right ).
\ee
It is straightforward to verify that
 \be
 (\vec{X} \cdot \bm{\rho})(\vec{Y} \cdot \bm{\rho}) = (\vec{X} \cdot_\eta \vec{Y}) \mathds{1} +  ( \vec{X} \wedge_\eta \vec{Y}) \cdot \bm{\rho}.
 \ee
 In particular, we find that $\bm{\tilde{S}}^2 = -\mathds{1}$ if and only if $\vec{S} \cdot_\eta \vec{S} = -1$. Now, we can recast \eqref{eq:hwmh} into the following form
 \be
 \pt_t \bm{\tilde{S}} = \frac{1}{2} [ \bm{\tilde{S}}, |\nabla| \bm{\tilde{S}} ].
 \ee
 
 \begin{theorem}[Lax Pair for Half-Wave Maps with Target $\HH^2$]
 Let $\vec{S}=\vec{S}(t,x)$ be a sufficiently regular solution of the half-wave maps equation \eqref{eq:hwmh}, where either $x \in \R$ or $x \in \T$. Then the following Lax equation holds true:
 $$
 \frac{d}{dt} L_{\bm{\tilde{S}}} = \ii [ B_{\bm{\tilde{S}}}, L_{\bm{\tilde{S}}} ],
 $$
 where the operators $L_{\bm{\tilde{S}}}$ and $B_{\bm{\tilde{S}}}$ are given by \eqref{def:Lax} with $\bm{S}$ replaced by $\bm{\tilde{S}}$.
 \end{theorem}
 Note the factor of $\ii$ on the right-hand side. In general, the Lax operator $L_{\bm{\tilde{S}}}$ is neither symmetric nor skew-symmetric anymore due to the fact that $\bm{\tilde{S}}$ is neither Hermitian nor anti-Hermitian in general. 
 
 \begin{proof}
 The proof of Theorem \ref{thm:lax} carries over mutatis mutandis. 
 \end{proof}
 
 \section{Summary and Conclusion}
 We have proved that the half-wave maps equation
 $$
\pt_t \vec{S} = \vec{S} \wedge |\nabla| \vec{S},
$$
where $\vec{S}=\vec{S}(t,x)$ is valued into the two--dimensional sphere $\Ss^2$ or the two--dimensional hyperbolic space $\HH^2$, and $x\in \R $ or $\T$, enjoys a Lax pair. The Lax operator is connected to Hankel operators with special matrix symbols. This allowed us to find new conservation laws equivalent to the homogeneous Besov norms $\| \cdot \|_{\dot B^{1/p}_{p,p}}$. As another consequence of the Lax equation, we establish that the subclass rational functions of $x\in \R $ on the line or of ${\rm e}^{\ii x}$ on the circle is conserved by the dynamics. More generally, it is expected that the inverse spectral theory of Hankel operators developed in \cite{GGAst}, \cite{Po} for studying dynamics of the cubic Szeg\H{o} equation will be of valuable help in the study of dynamics of this energy-critical half-wave maps equation. We hope to come back to these questions in a forthcoming paper.

\begin{appendix}
\end{appendix}


\begin{thebibliography}{Pe1}

	
\bibitem{BuLe} Bugiera, L. and Lenzmann, E.,  work in preparation. 

\bibitem{LR} Da Lio, F. and Rivi\`ere, T., {\em Three-term commutator estimates and the regularity of
              {$\frac12$}-harmonic maps into spheres}, Anal. PDE 4 (2011), 149--190.

\bibitem{FT} Faddeev, L. D. and Takhtajan, L. A., {\em Hamitonian methods in the theory of solitons}, Classics in Mathematics, Springer, Berlin, 2007.

\bibitem{FKL} Fr\"ohlich, J., Knowles, A. and Lenzmann, E., {\em Semi-classical dynamics in quantum spin systems}, Lett. Math. Phys. 82 (2007), 275--296.

\bibitem{GGAst} G\'erard, P., and Grellier, S., {\em The cubic Szeg\H{o} equation and Hankel operators}, Ast\'erique 389, Soci\'et\'e Math\'ematique de France, 2017.

\bibitem{Ha} Haldane, F. D. M., {\em Exact Jastrow-Gutzwiller resonating-valence-bond ground state of the spin-(1/2 antiferromagnetic Heisenberg chain with 1/${\mathrm{r}}^{2}$ exchange}, Phys. Rev. Lett. 60 (1988), 635--638.


\bibitem{La} Lakshmanan, M., {\em Continuum spin system as an exactly solvable dynamical system,} Phys. Lett. A 61 (1977), 53-.54,
\bibitem{LS} Lenzmann, E., and Schikorra, A., {\em On energy--critical half--wave maps into $\Ss^2$}, preprint, 2017, {\tt arXiv:1702.05995v2}.

\bibitem{Li}  Lieb, E. H.,  {\em The classical limit of quantum spin systems}, Comm. Math. Phys. 31 (1973), 327--340.
 
 \bibitem{Pe1} Peller, V. V., {\em Hankel operators of class $S_p$ and their applications (rational approximation, Gaussian processes, the problem of majorization of operators)}.  Mat. Sb.  113(155) (1980), no. 4(12), 538Ð581, 637. 

\bibitem{Pe2} Peller, V.V., {\em Hankel Operators and their applications}  Springer Monographs in Mathematics.
 Springer-Verlag, New York, 2003.

\bibitem{Po} Pocovnicu, O. {\em Explicit formula for the solution of the Szeg\H{o} equation on the real line and applications,} Discrete Cont. Dyn. Syst. 31 (2011), 607--649.

\bibitem{Sh} Shastry, B. S., {\em Exact solution of an S=1/2 Heisenberg antiferromagnetic chain with long-ranged interactions}, Phys. Rev. Lett. 60 (1988), 639--642.

\bibitem{Su} Sutherland, B., {\em Beautiful Models: 70 Years of Exactly Solved Quantum Many-Body Problems}, World Scientific, Singapore, 2004.

\bibitem{Ta} Takhtajan,  L.A. {\em Integration of the continuous Heisenberg spin chain through the inverse scattering method}, Phys. Lett. A 64 (1977), 235--237.

\bibitem{Zh} Zhou, T. and Stone, M., {\em Solitons in a continuous classical Haldane--Shastry spin chain,} Phys. Lett. A 379 (2015), 2817--2825.

\end{thebibliography}
\end{document}